\documentclass[reqno,11pt,twoside,english]{amsart}
\usepackage[T1]{fontenc}
\usepackage[latin9]{inputenc}
\usepackage{amsthm}
\usepackage{amssymb}
\usepackage{amsmath,amsfonts,epsfig}

\makeatletter

\numberwithin{equation}{section}

\usepackage{color}
\usepackage[final,linkcolor = blue,citecolor = blue,colorlinks=true]{hyperref}

\usepackage{type1cm,ae}
\usepackage{graphicx}
\usepackage{xspace}
\usepackage{setspace}
\usepackage[english]{babel}
\usepackage{tikz}
\usetikzlibrary{arrows,patterns}

 \newcommand{\R}{{\mathbb R}}

 \newcommand\dist{{\mathop\mathrm{\,dist\,}}}

 \newcommand\ep{\epsilon}

  \newcommand\Om{{\Omega}}

 \newcommand\diam{{\mathop\mathrm{\,diam\,}}}

 \newcommand*{\ssub}{\subset\subset}
\DeclareMathOperator{\Lip}{Lip}

\numberwithin{equation}{section}
\newtheorem{theorem}{Theorem}[section]

\newtheorem{corollary}[theorem]{Corollary}

\newtheorem{lemma}[theorem]{Lemma}

\theoremstyle{definition}
\begingroup
\newtheorem{definition}[theorem]{Definition}
\newtheorem{remark}[theorem]{Remark}

\endgroup

\def\Xint#1{\mathchoice
	{\XXint\displaystyle\textstyle{#1}}%
	{\XXint\textstyle\scriptstyle{#1}}%
	{\XXint\scriptstyle\scriptscriptstyle{#1}}%
	{\XXint\scriptscriptstyle\scriptscriptstyle{#1}}%
	\!\int}
\def\XXint#1#2#3{{\setbox0=\hbox{$#1{#2#3}{\int}$}
		\vcenter{\hbox{$#2#3$}}\kern-.5\wd0}}

\def\dashint{\Xint-}

\makeatother

\usepackage{babel}
\begin{document}

\title[Regularity of quasi-$n$-harmonic mappings into NPC spaces]{Regularity of quasi-$n$-harmonic mappings into NPC spaces}

\author{Chang-Yu Guo and  Chang-Lin Xiang*}

\address[Chang-Yu Guo]{Department of Mathematics, University of Fribourg, CH-1700, Fribourg, Switzerland}
\email{changyu.guo@unifr.ch}

\address[Chang-Lin Xiang]{School of Information and Mathematics, Yangtze University, Jingzhou 434023, P.R. China}
\email{changlin.xiang@yangtzeu.edu.cn}

\thanks{* Corresponding author}
\thanks{C.-Y. Guo was supported by Swiss National Science Foundation Grant 153599 and 165848. C.-L. Xiang is  financially supported by the National Natural Science Foundation of China (No. 11701045) and by  the Yangtze Youth Fund (No. 2016cqn56).}

\begin{abstract}
We prove local H\"older continuity of quasi-$n$-harmonic mappings from Euclidean domains into metric spaces with non-positive curvature in the sense of Alexandrov. We also obtain global H\"older continuity of such mappings from bounded Lipschitz domains.
\end{abstract}

\maketitle

{\small
\smallskip
\newline
}
\bigskip

\arraycolsep=1pt

\section{Introduction and main results}

\subsection{Background}

Given a mapping $u\colon M\to N$ between two  Riemannian manifolds with $\dim M=n$ and $1<p<\infty$, there is a natural concept of $p$-energy associated to $u$. Minimizers (or more generally, critical points) of such energy functionals are referred to  as $p$-harmonic mappings and harmonic mappings in case $p=2$. The research on harmonic mappings has a long and distinguished history, making it one of the most central topics in geometric analysis on manifolds~\cite{li12,sy97}. In his pioneering work, Morrey~\cite{m48} proved the H\"older continuity of an energy minimizing map when $n=2$ (and smooth if $M$ and $N$ smooth).  The breakthrough in higher dimensional theory of harmonic mappings was made by Eells and Sampson~\cite{es64},  where they proved that every homotopy class of maps from a closed manifold $M$ into $N$ has a smooth representative, with $N$  having nonpositive curvature. See more results in e.g. Hartman~\cite{h67} and Hamilton~\cite{h75}.  The regularity theory for harmonic mappings into general target Riemannian manifolds has later been developed by Schoen and Uhlenbeck in the seminal paper~\cite{su82}, and which obtained extension by Hardt and Lin~\cite{Hardt-Lin-1987} to  general $p$-harmonic mappings ($1<p<\infty$).

Inspired by the celebrated work of Gromov and Schoen~\cite{gs92}, where the authors proposed a variational approach for the theory of harmonic mappings to the setting of mappings into singular metric spaces and successfully applied to important rigidity problems for certain discrete groups, and by the fundamental work~\cite{ks93}, where the authors established existence, uniqueness and local Lipschitz regularity theory for harmonic mappings from compact smooth Riemannian manifolds to singular metric spaces, harmonic mappings into or between singular metric spaces has received considerable amount of growing interest during the last twenty years, with a particular emphasis on metric spaces of non-positive curvature in the sense of Alexandrov (NPC); see for instance~\cite{cl01,c95,dm10,ef01,g17,gw17,hz16, iw08,j94,j97,lw16,m98,st01,s02,s02b,zz13}.
In particular, in the research monograph of Eells-Fuglede~\cite{ef01}, the authors extended the theory of harmonic mappings $u:\Om\to X$ to the setting where $\Omega$ is an admissible Riemannian polyhedron and $X$ an NPC space. Gregori~\cite{g98} further extended the
existence and uniqueness theory of harmonic mappings to the setting where $X$ is a Lipschitz Riemannian manifold. Capogna and Lin~\cite{cl01} extended part of the harmonic mapping theories to the setting of mappings from Euclidean spaces to the Heisenberg groups. Sturm~\cite{st01,s02,s02b} developed a theory of harmonic mappings betwen singular metric spaces via a probabilistic theory and (generalized) Dirichlet forms.
In the recent remarkable work of Zhang and Zhu \cite{zz13}, the authors proved the important interior Lipschitz regularity of harmonic mappings from certain Alexandrov spaces to NPC spaces. Parallel to the mapping case, the theory of harmonic functions on singular metric spaces also gained growing interest in the last twenty years; see for instance~\cite{j14,ks01,krs03,s01} and the references therein.

Besides the harmonic case ($p=2$) and the intermediate case ($1<p<n$), the borderline case $p=n$ also received special attention as they often enjoy better property than general $p$-harmonic mappings. For instance, among other results, Hardt and Lin~\cite{Hardt-Lin-1987} showed that minimizing $n$-harmonic mappings from an $n$-dimensional compact Riemannian manifold into a $C^2$ Riemannian manifold are locally $C^{1,\alpha}$ for some $0<\alpha<1$;  Wang~\cite{w05} proved that $n$-harmonic mappings into Riemannian manifolds (without boundary) enjoy nice compactness properties; Mou and Yang~\cite{my96} obtained that $n$-harmonic mappings are everywhere regular in the interior, continuous up to the boundary (of a bounded smooth domain), and have removable isolated singularities; see also~\cite{nvv16} for a recent improvement of this result. When the target metric space is the real line $\R$, $n$-harmonic functions play a particularly important role in the theory of quasiconformal mappings and quasiregular mappings; see for instance~\cite{hkm06,io12} and the various references therein.

\subsection{Main results}
We first recall the following definition of quasi-$n$-harmonic mappings.

\begin{definition}(quasi-$n$-harmonic mappings)\label{def:quasi-n-harmonic}
Let $\Omega\subset \R^n$ be a domain and $X$ a metric space. A mapping $u\in W^{1,n}(\Omega,X)$ is said to be $Q$-quasi-$n$-harmonic, $Q\geq 1$, if
$$E_n(u|_{\Omega'})\leq Q\cdot E_n(v|_{\Omega'})$$
for every relatively compact domain $\Omega'\subset \Omega$ and every $v\in W^{1,n}(\Omega,X)$ with $u=v$ almost everywhere in $\Omega\backslash \Omega'$, where $E_n(u)$ is the $n$-energy of $u$ defined as in Section~\ref{subsec:Sobolev maps}.
\end{definition}

Note that 1-quasi-$n$-harmonic mappings are also called $n$-harmonic mappings in literature (see e.g. Hardt and Lin \cite{Hardt-Lin-1987}). When $n=2$, we recover the class of quasi-harmonic or harmonic mappings. Another typical nontrivial example in higher dimensions is given by quasiregular mappings (see Giaquinta and Giusti \cite[Theorem 2.4]{gg84}), which is a mapping $u:\Omega\to \mathbb{R}^n$ satisfying $|Du|^n\le k\det(Du)$. In particular, when $u$ is a homeomorphism, it is a quasicomformal mapping. More generally, quasi-$n$-harmonic mappings is a special case of quasi-minima that was initially studied by Giaquinta and Giusti \cite{gg84} in Euclidean spaces.

Based on the recent solution of Plateau's problem in proper metric spaces~\cite{lw17}, Lytchak and Wenger~\cite{lw16} considered the interior regularity of quasi-harmonic mappings from two-dimensional Euclidean domains to proper metric spaces. They proved that each quasi-harmonic mapping $u\colon \Omega\to X$ from a planar Euclidean domain to a large class of proper metric spaces has a locally H\"older continuous representative.

Motivated by the above work of Lytchak and Wenger~\cite{lw16}, and also by the recent development of harmonic mappings in singular metric spaces, in this short note,  we study interior and boundary regularity of quasi-$n$-harmonic mappings from Euclidean domains to NPC spaces.

Our first main result can be viewed as a natural partial extension of the interior regularity result of Lytchak and Wenger~\cite{lw16} to higher dimensions.

\begin{theorem}\label{thm:main theorem}
Let $\Omega\subset \R^n$ be a domain and $X$ an NPC space. Then each $Q$-quasi-$n$-harmonic mapping $u\colon \Omega\to X$ has a locally $\alpha$-H\"older continuous representative for some $\alpha$ depending only on $Q$ and $n$.
\end{theorem}

We would like to point out the the H\"older continuity in Theorem~\ref{thm:main theorem} is best possible even when $X=\R$; see~\cite{s01,ks01}. As a corollary of Theorem~\ref{thm:main theorem}, we obtain that each quasi-harmonic mapping from planar Euclidean domains to NPC spaces has a locally H\"older continuous representatives.

\begin{corollary}\label{coro:quasiharmonic discs} Let $\Omega\subset \R^2$ be a domain and $X$ an NPC space. Then each $Q$-quasi-harmonic mapping $u\colon \Omega\to X$ has a locally $\alpha$-H\"older continuous representative for some $\alpha$ depending only on $Q$.
\end{corollary}

Corollary~\ref{coro:quasiharmonic discs} is not really new, and in fact, it follows from the proof of~\cite[Theorem 1.3]{lw16}. Indeed, the main ingredients in their arguments are \emph{solvability of Plateau problem in proper metric spaces} and \emph{an energy filling inequality} (i.e. \cite[Theorem 1.5]{lw16}). The solvability of Platau problem is well-known in the context of NPC spaces (see e.g.~\cite{gw17}) and the authors also pointed out the energy filling inequality holds for general NPC spaces. Thus Corollary~\ref{coro:quasiharmonic discs} follows from the proof of Theorem 1.3 there. However, our proof of Corollary~\ref{coro:quasiharmonic discs} is more elementary and simpler, comparing with the more general proof there. On the other hand, our proof relies heavily on the special structure of NPC spaces and hence seems hard to be extended to the more general setting as considered in~\cite{lw16}. As to quasiharmonic mappings on higher dimensional Euclidean domains ($n\ge 3$),  there is no hope to derive H\"older continuity in this respect, even for quasiharmonic mappings into Euclidean domains (in this case,  quasiharmonic mappings are also named quasiminima). As pointed out in Giaquinta \cite[page 253]{Giaquinta-Book}, there exists a quasiminima for some Dirichlet integral which is singular in a dense set.

Our second main result concerns boundary regularity of quasi-$n$-harmonic mappings from bounded Lipschitz domain to NPC spaces, which can be viewed as a natural partial extension of~\cite[Theorem 1.4]{lw16}. There have been extensive contributions for boundary regularity on harmonic mappings in the literature which is impossible to list completely. We only mention a few that are most related to our work.  For mappings from bounded smooth Euclidean domain to  Euclidean spaces, we refer to Jost and Meier \cite{Jost-Meier-83-MathAnn} in which more general result was obtained. That is, bounded minimum of certain quadratic functional is proven to be H\"older continuous in a neighborhood of the boundary with sharp H\"older exponent. For mappings from a compact Riemannian manifold with boundary or from bounded smooth domain of Riemannian manifolds to smooth Riemannian manifolds, we refer to Schoen and Uhlenbeck \cite{su83} for harmonic mappings and Hardt and Lin \cite{Hardt-Lin-1987} for $p$-harmonic mappings, respectively. As for boundary regularity results on harmonic mappings from compact Riemannian domains to NPC, we would like to refer to the work of Serbinowski \cite{Serbinowski-94-CAG}. Our result is new in the setting of quasi-$n$-harmonic mappings.

\begin{theorem}\label{thm:boundary regularity}
Let $\Omega\subset \R^n$ be a bounded domain with a Lipschitz boundary $\partial \Omega$ and $X$ an NPC space.
Let $u\colon \Omega\to X$ be a $Q$-quasi-$n$-harmonic mapping whose trace coincides with the trace of $h\in W^{1,p}(\Omega,X)$ with $p>n$. Then $u$ is $\alpha$-H\"older continuous in a neighborhood of $\partial\Omega$ for some $\alpha$ depending only on $Q$, $n$ and $p$.	
\end{theorem}

Above, the trace of $u\in W^{1,n}(\Omega,X)$ coincides with the trace of $g\in W^{1,p}(\Omega,X)$ is equivalent to the requirement that $d(u,g)\in W^{1,n}_0(\Omega)$ (see e.g.~\cite[Section 1.12]{ks93}). We would like to point out that in~\cite[Theorem 1.4]{lw16}, the trace of $u$ was required to be Lipschitz continuous, which is a little bit stronger than what we have assumed in Theorem~\ref{thm:boundary regularity}.

\subsection{Outline of proof}\label{subsec:outline of proof}
 Our main tool to prove Theorem~\ref{thm:main theorem} and Theorem~\ref{thm:boundary regularity} is the  reverse H\"older inequality, which was discovered by Gehring \cite{Gehring-1973-Acta} in his celebrated work on higher regularity of quasiconformal mappings and was later developed by  Giaquinta and Modica \cite{Giaquinta-Modica-1979-JRAM} (see also \cite{Giaquinta-Book}) and many others in the theory of  elliptic partial differential equations. In this note, we will use the following local type reverse H\"older inequality; see Proposition 5.1 of  \cite{Giaquinta-Modica-1979-JRAM} or  Proposition 1.1 of Chapter V of \cite{Giaquinta-Book}.

Denote by $Q_R\subset\R^n$ a cube with side-length $R$ and let $q>1$. Let $g\in L^q_{\rm loc}(Q_1)$ and $f\in L^r_{\rm loc}(Q_1)$ ($r>q$) be two nonnegative functions.  Suppose there exist constants $b>1$ and $\theta\in [0,1)$, such that   for every $x_0\in Q_1$ and $2R<\dist(x_0,\partial Q_1)$ the following estimate holds
\begin{equation}\label{eq:key estimate}
\dashint_{Q_R(x_0)}g^q{d}x\le b\left\{ \left(\dashint_{Q_{2R}(x_{0})}g{ d}x\right)^{q}+\dashint_{Q_{2R}(x_{0})}f^{q}{ d}x\right\} +\theta\dashint_{Q_{2R}(x_{0})}g^{q}{ d}x,
\end{equation}
where $\dashint_A f{\rm d}x:=|A|^{-1}\int_A f{\rm d}x$. Then, there exist $\ep>0$ and $C>0$, depending only on $\theta, b,q,n$, such that $g\in L^p_{\rm loc}(Q)$ for $p\in [q,q+\ep)$ and
\begin{equation}\label{eq: RHI}\left(\dashint_{Q_{1/2}}g^p { d}x\right)^{1/p}\le C\left\{ \left(\dashint_{Q}g^q {d}x\right)^{1/q} + \left(\dashint_{Q}f^p { d}x\right)^{1/p}  \right\}.\end{equation}

In our arguments, to derive estimates of type ~\eqref{eq:key estimate} with suitable choices $g$ and $f$, we will borrow the idea from \cite{gg84}. More precisely, we will compare $u$ with ``$u_\eta=(1-\eta)u+\eta p_0$" for $0\le \eta\le 1$, which is well-defined in NPC spaces. The main new ingredient in our proof is certain new estimates on  pullback tensors.

Our notations are rather standard. We will use $C$ or $c$ to denote various  constants that may be different from line to line.

\section{Preliminaries}

\subsection{Sobolev mappings with value in metric spaces}\label{subsec:Sobolev maps}
Let $\Omega$ be a domain in $\R^n$ and $(X,d)$  a metric space. We follow  Korevaar and Schoen \cite{ks93} to define Sobolev mappings $u\colon \Om\to X$. Given $p>1$, $\epsilon>0$, we define
\[e_{p,\epsilon}^u(x)=\dashint_{B_{\epsilon}(x)} \frac{d^p(u(x),u(y))}{\epsilon^{p}}dy\]
for all $x\in \Omega_{\varepsilon}:=\{z\in \Omega: d(z,\partial \Omega)>\varepsilon \}$ and $e_{p,\epsilon}^u(x)=0$ for all $x\in \Omega\backslash \Omega_\varepsilon$. For each $u\in L^p(\Om,X)$, we define the approximate energy
\[E_{p,\epsilon}^u(f)=c(n,p)\int_{\Om}f(x)e_{p,\epsilon}^u(x){ d}x,\qquad f\in C_c(\Om),\]
where  $C_c(\Om)$ consists of continuous functions in $\Om$ with compact support and $c(n,p)>0$ is a normalization constant. Then, $u$ is said to have finite $p$-energy, written as $u\in W^{1,p}(\Om,X)$, if
\[E_{p}(u)\equiv \sup_{f\in C_c(\Om),0\le f\le 1}\limsup_{\epsilon\to 0}E_{p,\epsilon}^u(f)<\infty.\]

By \cite[Theorem 1.5.1]{ks93}, if $u\in W^{1,p}(\Omega,X)$, then the measures $e^u_{p,\varepsilon}dx$ converges weakly as $\varepsilon\to 0$ to an energy density measure $de^u_p$ with total measure $E_p(u)$. Moreover, by~\cite[Theorem 1.10]{ks93}, $de^u_p$ is absolutely continuous with respect to the Lebesgue measure. In particular, there exists $|\nabla u|_p\in L^1(\Omega,\R)$ such that
$$de_p^u=|\nabla u|_pdx.$$
When $p=2$, we write $|\nabla u|^2$ instead of $|\nabla u|_2$. Note that in general
$$|\nabla u|_p\neq \big(|\nabla u|^2\big)^{p/2}$$
but they are comparable up to a uniform constant.

There are also many other equivalent definitions of metric-valued Sobolev spaces and we recommend the interested readers to~\cite{hkst12} for more information. As a special consequence of the equivalence with Newtonian-Sobolev spaces, we have the following Sobolev-Poincar\'e inequality for Sobolev mappings: 
\begin{lemma}\label{lemma:Sobolev embedding}
For  $1<p<n$, there exists a positive constant $c(n,q,p)$ such that for each $p<q<p^*=\frac{np}{n-p}$
\begin{equation}\label{eq:Sobolev Poincare}
\inf_{a\in X}\Big(\dashint_B d^{q}(u,a)dx\Big)^{1/q}\leq c(n,p,q)\diam B\Big(\dashint_{B}|\nabla u|_p dx\Big)^{1/p}
\end{equation}
holds for every $u\in W^{1,p}(\Omega,X)$ and  every ball $B$ with $4B\ssub\Omega$.
\end{lemma}

\begin{proof}
The proof is essentially the same as~\cite[Proof of Theorem 3.6]{kst04} (whereas the idea dates back to~\cite{hkst01}). For the convenience of the readers, we include the main steps here. Embed $X$ isometrically in the Banach space $\mathbb{V}=l^\infty(X)$ such that $X\subset \mathbb{V}=(\mathbb{V},\|\cdot\|)$. Let $\Lambda\in V^*$ be such that $L:=\|\Lambda\|_{V^*}\leq 1$. Then $\Lambda\colon \mathbb{V}\to \R$ is an $L$-Lipschitz map and $f:=\Lambda\circ u\in W^{1,p}(\Omega)$. Moreover, the standard Sobolev-Poincar\'e inequality for $f$ implies that for each $q<p^*$,
\begin{equation*}
\dashint_{B}|f-f_{B}|^{q}dx\leq c(n,p,q)L^{q}\big(\diam B\big)^{q}\Big(\dashint_{B}|\nabla u|_p dx\Big)^{q/p},
\end{equation*}
where we have also used the fact that
$$|\nabla f|^p\leq \big(\Lip \Lambda\big)^p\cdot g_u\leq c(n,p)L^p|\nabla u|_p$$
holds almost everywhere on $B$, where $g_u$ is the minimal $p$-weak upper gradient of $u$ (see~\cite{hkst12} for precise definition) (and the last inequality comes from the equivalence of Newtonian-Sobolev spaces~\cite[Section 10.4]{hkst12}).

If $x,y$ are Lebesgue points of $f$, then letting $B_0=B(x,2|x-y|)$, $B_i=B(x,2^{-i}|x-y|)$, $B_{-i}=B(y,2^{-i}|x-y|)$ and using the standard telescoping argument, we obtain the following useful pointwise inequality for $f$:
\begin{align*}
|\Lambda\circ u(x)-\Lambda\circ u(y)|&\leq \sum_{i\in \mathbb{Z}}|f_{B_i}-f_{B_{i+1}}|\\
&\leq c(n,p)L|x-y|\Big(M_{2|x-y|}|\nabla u|_p(x)^{1/p}+M_{2|x-y|}|\nabla u|_p(y)^{1/p}\Big),
\end{align*}
where
$$M_{R}|\nabla u|_p(x)=\sup_{0<r<R}\dashint_{B(x,r)}|\nabla u|_p(z                                                                                                                                                                                                                                                                                                                                                                                                                                                                                                                                                                                                                                                                                                                                                                                                                                                                                                                                                                                                                                                                                                                                                                                                                                                                                                                                                                                                                                                                                                                                                                                                                                                                                                                                                                                                                                                                                                                                                             )dz$$
is the standard restricted maximal function of $|\nabla u|_p$.

The next step is to show that for almost every $x,y\in B$ we have 
\begin{equation}\label{eq:Haljasz inequality}
\|u(x)-u(y)\|\leq c(n,p)|x-y|\Big(M_{2|x-y|}|\nabla u|_p(x)^{1/p}+M_{2|x-y|}|\nabla u|_p(y)^{1/p}\Big).
\end{equation}
In this step, one can follow the arguments used in~\cite[the last paragraph in the proof of Theorem 3.6]{kst04} word by word. In fact, only the previous pointwise inequality for $f$ is needed.

The final step is to show that the pointwise inequality~\eqref{eq:Haljasz inequality} implies the following Sobolev-Poincar\'e inequality: for each $q<p^*$,
\begin{align}\label{eq:aim}
\inf_{a\in Y}\dashint_{B}d(u(x),a)^{q}dx\leq c(n,p,q)\big(\diam B\big)^{q}\Big(\dashint_{B}|\nabla u|_p dx\Big)^{q/p}.
\end{align}
The proof of this is more or less well-known (see~\cite[Proof of Proposition 3.12]{kst04} or the monograph~\cite[Section 9.1]{hkst12}) and we include a sketch here for the convenience of the readers. First, note that~\eqref{eq:Haljasz inequality} together with H\"older's inequality implies that
\begin{equation}
\dashint_{B}\|u(x)-u_B\|dx\leq c(n,p,q)\diam B \Big(\dashint_{B}M_{2\diam B}|\nabla u|_pdx\Big)^{1/p}.
\end{equation}
Fix $0<\varepsilon<1$ (to be determined later), and, for $t>0$, let 
$$A_t=\big\{x\in B: \|u(x)-u_B\|>t\big\}. $$
For each $i\in \mathbb{N}$, set $B_i=\{z\in \Omega: |x-z|<2^{-i}\diam B \}$. It is clear that $B_i\subset 2B$. At each Lebesgue point $x\in A_t$ of the map $u\colon X\to Y\subset \mathbb{V}$, we have
\begin{align*}
C(\varepsilon)t\sum_{i\in \mathbb{N}}2^{-i(1-\varepsilon)}&=t<\|u(x)-u_B\|\leq \sum_{i\in \mathbb{N}}\|u_{B_{i+1}}-u_{B_i}\|\\
&\leq c(n,p,q)\dashint_{B_i}\|u(z)-u_{B_i}\|dz\\
&\leq c(n,p,q)\diam B\sum_{i\in \mathbb{N}}2^{-i}\Big(\dashint_{B_i}M_{2\diam B_i}|\nabla u|_pdx\Big)^{1/p}.
\end{align*}
Hence there exists a positive integer $i_x$ such that
$$C(\varepsilon)t2^{-i_x(1-\varepsilon)}\leq c(n,p,q)\diam B\cdot 2^{-i_x}\Big(\dashint_{B_{i_x }}M_{2\diam B_{i_x}}|\nabla u|_pdx\Big)^{1/p},$$
or equivalently,
$$\mathcal{L}^n(B_{i_x})\leq c(n,p,q)^s\big(\frac{\diam B}{t}\big)^{ps}\frac{\Big(\int_{B_{i_x }}M_{2\diam B_{i_x}}|\nabla u|_pdx\Big)^s}{\mathcal{L}^n(B)^{s-1}}, $$
where $s=\frac{n}{n-p\varepsilon}>1$ and $\varepsilon<1$ is a fixed small number. Using the $(5B)$-covering lemma, we easily obtain that
$$\mathcal{L}^n(A_t)\leq c(n,p,q)^s\big(\frac{\diam B}{t}\big)^{ps}\mathcal{L}^n(B)^{1-s}\Big(\int_{4B}M_{4\diam B}|\nabla u|_pdx\Big)^s.$$
Set 
$$C_0:= c(n,p,q)^s\big(\diam B\big)^{ps}\Big(\int_{4B}M_{4\diam B}|\nabla u|_pdx\Big)^s.$$ Then
$\mathcal{L}^n(A_t)\leq C_0\frac{\mathcal{L}^n(B)^{1-s}}{t^{ps}}$ and an easy application of the Cavalier's principle (see \cite[Proof of Lemma 3.23]{kst04}) gives 
\begin{equation*}
\int_{B}\|u(x)-u_B\|^qdx\leq C_0^{\frac{q}{ps}}\big(\frac{q}{q-ps}\big)^{q/(ps)}\mathcal{L}^n(B)^{1-\frac{q}{p}},
\end{equation*}
which reduces to the desired inequality~\eqref{eq:aim} upon noticing the $L^p$-boundedness of the maximal operator.

\end{proof}
 
\begin{remark}\label{rmk:on good p}
In the case $1<p<n$, Lemma~\eqref{lemma:Sobolev embedding} holds with $q=p^*$ as well. This can be proved by a truncation argument due to Hajlasz and Koskela. Since this stronger case is not needed for the current paper, we do not include the proof here. 

In the borderline case $p=n$, one can similarly prove that
$$\dashint_B\exp\Big(\Big(\frac{d(u,u_B)}{c_1(n)\diam B(\dashint_B |\nabla u|_ndx)^{1/n} } \Big)^{n/(n-1)} \Big)\leq c_2(n)$$ for  $u\in W^{1,n}(\Omega,X)$.

In the case $p>n$, one can similarly prove that each $u\in W^{1,p}(\Omega,X)$  has a locally H\"older continuous representative (see also~\cite[Proposition 3.3]{lw17}).
\end{remark}

\subsection{Metric spaces with non-positive curvature in the sense of Alexandrov}
\begin{definition}[NPC spaces]
A complete metric space $(X,d)$
(possibly infinite dimensional) is said to be non-positively curved in the sense of Alexandrov
(NPC) if the following two conditions are satisfied:

\begin{itemize}
	\item $(X,d)$ is a length space, that is, for any two points $P,Q$ in
	$X$, the distance $d(P,Q)$ is realized as the length of a rectifiable
	curve connecting $P$ to $Q$. (We call such distance-realizing curves
	geodesics.)
	\item For any three points $P,Q,R$ in $X$ and choices of geodesics $\gamma_{PQ}$
	(of length $r$), $\gamma_{QR}$ (of length $p$), and $\gamma_{RP}$ (of
	length $q$) connecting the respective points, the following comparison
	property is to hold: For any $0<\lambda<1$, write $Q_{\lambda}$ for the
	point on $\gamma_{QR}$ which is a fraction $\lambda$ of the distance from
	$Q$ to $R$. That is,
	\[
	d(Q_{\lambda},Q)=\lambda p,\quad d(Q_{\lambda},R)=(1-\lambda)p.
	\]
	On the (possibly degenerate) Euclidean triangle of side lengths $p,q,r$
	and opposite vertices $\bar{P}$, $\bar{Q},\bar{R}$, there is a corresponding
	point
	\[
	\bar{Q}_{\lambda}=\bar{Q}+\lambda(\bar{R}-\bar{Q}).
	\]
	The NPC hypothesis is that the metric distance $d(P,Q_{\lambda})$ (from
	$Q_{\lambda}$ to the opposite vertex $P$) is bounded above by the Euclidean
	distance $|\bar{P}-\bar{Q}_{\lambda}|$. This inequality can be written
	precisely as
	\[
	d^{2}(P,Q_{\lambda})\le(1-\lambda)d^{2}(P,Q)+\lambda d^{2}(P,R)-\lambda(1-\lambda)d^{2}(Q,R).
	\]
\end{itemize}
\end{definition}

In an NPC space $X$, geodesics connecting each pair of points are unique and so one can define the $t$-fraction mapping $u_t$ of two mapping $u_0,u_1\colon \Omega\to X$ as $u_t=``(1-t)u_0+tu_1"$, that is, for each $x$, $u_t(x)$ is the unique point $P$ on the geodesic connecting $u_0(x)$ and $u_1(x)$ such that $d(P,u_0(x))=td(u_0(x),u_1(x))$ and $d(P,u_1(x))=(1-t)d(u_0(x),u_1(x))$. We refer the interested readers to~\cite[Section 2.1]{ks93} for more discussions on NPC spaces.

\subsection{Pullback tensors}
Let $\Om\subset \R^n$ be a domain and $X$ an NPC space. For each $u\in W^{1,2}(\Om, X)$ and for any Lipschitz vector fields $Z,W$ on $\overline{\Om}$,  $u$ induces an integrable directional energy functional $|u_*(Z)|^2$, and moreover, Korevaar and Schoen \cite[Lemma 2.3.1]{ks93} proved the following important parallelogram identity
\[|u_{\ast}(Z+W)|^2+|u_{\ast}(Z-W)|^2=2|u_{\ast}(Z)|^2+2|u_{\ast}(W)|^2.\]  This property induces a pullback tensor $\pi=\pi_u$ over the  Lipschitz vector fields on  $\overline{\Om}$ by setting
\[\pi(Z,W)=\frac 14 |u_{\ast}(Z+W)|^2- \frac 14 |u_{\ast}(Z-W)|^2.\] It was proved in \cite[Theorem 2.3.2]{ks93}) that $\pi$  is continuous, symmetric, bilinear, nonnegative and tensorial. The pullback tensor generalizes the classical pullback metric $u^*h$ for mappings into Riemannian manifolds $(N,h)$ and plays a fundamental role in understanding the structure of harmonic mappings to NPC spaces in~\cite{ks93}.

\section{Proof of main results}
First we derive the interior H\"older continuity for quasi-$n$-harmonic mappings.

\begin{proof}[Proof of Theorem~\ref{thm:main theorem}]
	
Let $B_R\ssub \Omega$ be a ball of radius $R$ and let $\frac{R}{2}<t<s<R$. Let $\eta\in C_0^\infty(B_R)$ be such that $0< \eta<1$ on $B_s\backslash \bar{B}_t$, $\eta\equiv 0$ outside $B_s$, $\eta\equiv 1$ on $\bar{B}_t$, $|D\eta|\leq c(s-t)^{-1}$, $\{\eta=\frac{1}{2}\}$ has zero Lebesgue measure and $\int_{B_R}\frac{1}{(1-2\eta(x))^4}dx<\infty$.
We will compare the $n$-energy of $u$ with the function $u_\eta:=``(1-\eta(x))u(x)+\eta(x)p_0"$, where $p_0$ is chosen such that
$$\Big(\dashint_{B_R}d^n(u(x),p_0)dx\Big)^{1/n}\leq 2\inf_{a\in X}\Big(\dashint_{B_R}d^n(u(x),a)dx\Big)^{1/n}.$$

We now show that for any smooth function $\eta\in C_0^\infty(\Omega)$ which satisfies either $0\leq \eta<\frac{1}{2}$ or $\frac{1}{2}<\eta\leq 1$, it holds
\begin{equation}\label{eq:*}
\begin{aligned}
\pi_{u_\eta}&\leq \pi_{u}-\mathcal{C}(u_0,p_0,\eta)-\nabla\eta\otimes \nabla d^2(u,p_0)+Q(\eta,\nabla \eta)-\pi_{u_{1-\eta}}\\
&\leq (1-\eta)\pi_{u}+C|\nabla \eta|d(u,p_0)|\nabla u|_1-\nabla\eta\otimes \nabla d^2(u,p_0)+Q(\eta,\nabla\eta),
\end{aligned}\end{equation}
where $$\mathcal{C}(u,p_0,\eta)=\pi_u-P(u,p_0,\eta)-P(u,p_0,1-\eta)$$ is the auxiliary tensor defined in (2.4xiv) of~\cite{ks93} and
$$Q(\eta,\nabla \eta)=C\frac {|\nabla \eta(x)|^2 d^2(u(x),p_0)}{(1-2\eta(x))^2}$$ is quadratic in terms of $\eta$ and $\nabla \eta$ defined as in~\cite[Lemma 2.4.2]{ks93}, for some constant $C$ which may be different from line to line. For this, we first prove the following estimate:
\begin{equation}\label{eq:1}
	\pi_{u_{1-\eta}}\geq \eta\pi_u-\mathcal{C}(u,p_0,\eta)-C|\nabla \eta|d(u,p_0)|\nabla u|_1-Q(\eta,\nabla \eta).
\end{equation}
We first consider the case $0\leq \eta<\frac{1}{2}$. In this case, by  \cite[(2.4xvii)]{ks93}, on $\{\eta>0\}$ we have
\begin{align*}
\pi_{u_{1-\eta}}\geq P(u,p_0,1-\eta)-C|\nabla \eta|d(u,p_0)|\nabla u|_1-Q(\eta,\nabla \eta),
\end{align*}
where $C>0$ depends on $n$ and $P(u,p_0,1-\eta)$ is a symmetric bilinear integrable tensor defined in~\cite[Lemma 2.4.4]{ks93}.
Moreover, by~\cite[(2.4xvi)]{ks93}, we have
$$\eta\pi_u-\mathcal{C}(u,p_0,\eta)\leq P(u,p_0,1-\eta),$$
from which \eqref{eq:1} follows. The desired equation~\eqref{eq:*} follows by combining \eqref{eq:1} with~\cite[(2.4xv)]{ks93}. The case $\frac{1}{2}<\eta\leq 1$ can be proved similarly. Indeed, the only difference in this case would be the equation (2.4vii) of~\cite{ks93}, where one needs to replace the negative term $1-2\eta(y)$ by the positive term $2\eta(y)-1$. Then the equations (2.4xv), (2.4xvi) and (2.4xvi) hold (with the same proof as in~\cite[Lemma 2.4.5]{ks93}). Consequently, \eqref{eq:*} holds by the same arguments as in the first case. Sine $u\in W^{1,n}(B_R,X)$, we have $d(u,p_0)\in L^{2n}(B_R)$ (indeed, in any $L^p$ with $p<\infty$), this together with our choice that $\frac{1}{(1-2\eta)^2}\in L^2(B_R)$, implies that $Q(\eta,\nabla \eta)^\frac{n}{2}\in L^s(B_R)$ for some $s=s(n)>1$ (by the H\"older's inequality and by the proof of Lemma 2.4.2 in \cite{ks93}).

Set $\mu=n/2$. Taking trace on both sides of equation ~\eqref{eq:*}, (for simplicity and without any confusion, we denote the trace of $\pi$ by the same symbol), and then taking the $\mu$-th power on both sides, we get
\begin{align*}
	\pi_{u_\eta}^\mu\leq c(n)\Big( (1-\eta)^\mu\pi_{u}^\mu+|\nabla \eta|^\mu d^\mu(u,p_0)|\nabla u|_\mu-(\nabla \eta\otimes \nabla d^2(u,p_0))^\mu +Q(\eta,\nabla\eta)^\mu\Big).
\end{align*}
 Write  $g=Q(\eta,\nabla \eta)^\mu$. Then, $g\in L^s(B_R)$ for some $s>1$. Since $u$ is $Q$-quasi-$n$-harmonic,  we obtain from the above inequality that
\[
\begin{aligned}\frac{1}{c(n)} & \int_{B_{s}}(|\nabla u|^{2})^{\mu}dx\leq\int_{B_{s}}|\nabla u|_{n}dx\leq Q\int_{B_{s}}|\nabla u_{\eta}|_{n}dx\leq C\int_{B_{s}}\big(|\nabla u_{\eta}|^{2}\big)^{\mu}dx\\
 & \leq C\left(\int_{B_{s}}(1-\eta)^{\mu}(|\nabla u|^{2})^{\mu}dx+\int_{B_{s}}|\nabla\eta|^{\mu}\left(d^{\mu}(u,p_{0})|\nabla u|_\mu+|\nabla d^{2}(u,p_{0})|^{\mu}\right)dx+\int_{B_{s}}gdx\right)\\
 & \leq C\left(\int_{B_{s}-B_{t}}(|\nabla u|^{2})^{\mu}dx+\frac{1}{(s-t)^{\mu}}\int_{B_{R}}\left(d^{\mu}(u,p_{0})|\nabla u|_\mu+|\nabla d^{2}(u,p_{0})|^{\mu}\right)dx+\int_{B_{R}}gdx\right)\\
 & \leq C\Big(\int_{B_{s}-B_{t}}(|\nabla u|^{2})^{\mu}dx+\frac{1}{(s-t)^{\mu}}\int_{B_{R}}d^{\mu}(u,p_{0})|\nabla u|_{\mu}dx+\int_{B_{R}}gdx\Big)
\end{aligned}
\] for some constant $C>0$.
By H\"older's inequality and Young's inequality, we have for any
$\epsilon>0$
\[
\frac{1}{(s-t)^{\mu}}\int_{B_{R}}d^{\mu}(u,p_{0})|\nabla u|_{\mu}dx\le\frac{C_{\epsilon}}{(s-t)^{n}}\int_{B_{s}}d^{n}(u,p_{0})dx+\epsilon\int_{B_{s}}(|\nabla u|^{2})^{\mu}dx.
\]
Thus, by taking $\epsilon$ suitably small, we derive from the above
that
\[
\begin{aligned}\int_{B_{s}}(|\nabla u|^{2})^{\mu}dx  \le C_{0}\int_{B_{s}-B_{t}}(|\nabla u|^{2})^{\mu}dx
  +C_{0}\left\{ \frac{1}{(s-t)^{n}}\int_{B_{s}}d^{n}(u,p_{0})dx+\int_{B_{R}}gdx\right\}
\end{aligned}
\]
holds for some $C_{0}>0$. Adding $C_{0}\int_{B_{t}}(|\nabla u|^{2})^{\mu}dx$
on both sides of the above inequality and then dividing by $1+C_{0}$,
we deduce
\[
\begin{aligned}\int_{B_{t}}(|\nabla u|^{2})^{\mu}dx \le\theta\int_{B_{s}}(|\nabla u|^{2})^{\mu}dx +C\left\{ \frac{1}{(s-t)^{n}}\int_{B_{s}}d^{n}(u,p_{0})dx+\int_{B_{R}}gdx\right\}
\end{aligned}
\]
for some $C>0$, where $\theta=C_{0}/(1+C_{0})<1$.

Applying Lemma 3.2 of~\cite{gg84}, we obtain
\begin{equation}\label{eq: growth estimate}
\begin{aligned}\int_{B_{R/2}}(|\nabla u|^{2})^{\mu}dx & \le C\Big(R^{-n}\int_{B_{R}}d^{n}(u,p_{0})dx+\int_{B_{R}}gdx\Big).\end{aligned}
\end{equation}
By the Sobolev-Poincar\'e inequality (Lemma~\ref{lemma:Sobolev embedding}), for $\frac{n}{2}< q<n$, we get
\[
R^{-n}\int_{B_{R}}d^{n}(u(x),p_{0})dx\le C\Big(\int_{B_{R}}(|\nabla u|^{2})^{\frac{q}{2}}dx\Big)^{n/q}|B_{R}|^{1-\frac nq}.
\]
Combining  the above  estimate and \eqref{eq: growth estimate}, we infer that
\[
\int_{B_{R/2}}(|\nabla u|^{2})^{\mu}dx  \le C \Big\{\Big(\int_{B_{R}}(|\nabla u|^{2})^{\frac{q}{2}}dx\Big)^{n/q}|B_{R}|^{(q-n)/q}+\int_{B_{R}}gdx\Big\}.
\]
Now, set $w=\big((|\nabla u|^{2})^{\mu}\big)^{q/n}$, $\frac{n}{2}< q<n$.
It follows, for any $B_{R}\subset\subset\Om$,
that
\begin{equation}\label{eq: RHI for w}
\dashint_{B_{R/2}}w^{\frac{n}{q}}dx\leq C\left(\left(\dashint_{B_{R}}wdx\right)^{\frac{n}{q}}+\dashint_{B_{R}}gdx\right).
\end{equation}
As commented in Section~\ref{subsec:outline of proof}, the conclusion of  Theorem \ref{thm:main theorem}  follows from the local type reverse H\"older inequality \eqref{eq: RHI}. The proof is complete.
\end{proof}

Next we derive the boundary regularity for quasi-$n$-harmonic mappings. In \cite[Lemma 1]{Jost-Meier-83-MathAnn}, using boundary type reverse H\"older's inequality, Jost and Meier  improved the integrability of gradients of  local minima  for certain quadratic functionals $f:\Omega\times\mathbb{R}^N\times\mathbb{R}^{nN}$ close to the boundary. We will generalize the arguments in our setting.

\begin{proof}[Proof of Theorem~\ref{thm:boundary regularity}]
Let $x_0\in \partial \Omega$ and $V$ an open neighborhood of $x_0$. Since $\Omega$ is Lipschitz and quasi-$n$-harmonic mappings are stable under bi-Lipschitz transformations, we may perform a local bi-Lipschitz coordinate transformation such that $x_0$, $V\cap \Omega$ and $V\cap \partial \Omega$ get mapped onto $0$, $B_1^+$, and $\Gamma_1$, respectively.   Here
we  denote by $B_R^+$ the open half ball $\{x=(x_1,\cdots,x_n)\in \R^n:|x|<R,x_n>0 \}$ and $\Gamma_R=\{x=(x_1,\cdots,x_n)\in \R^n: |x|<R,x_n=0 \}$.
It suffices to show that $u\in W^{1,q}(B_{1/2}^+,X)$ for some $n<q=q(Q,n,p)\leq p$.

Fix any $R<1$. We will show that if $x\in B_R^+\cup \Gamma_R$ and $r<1-R$, then
\begin{equation}\label{eq: boundary RHI}
\begin{aligned}
	\Big(\dashint_{B_{r/2}(x)\cap B_R^+}|\nabla u|_{\tilde{p}}dx\Big)^{1/\tilde{p}}\leq c\Big(\dashint_{B_{r}(x)\cap B_R^+}|\nabla u|_ndx\Big)^{1/n}+c\left(\dashint_{B_{r}(x)\cap B_R^+}(|\nabla h|_p+g)dx\right)
\end{aligned}\end{equation} for some $\tilde{p}$ with $p>\tilde{p}>n$.

The argument is quite similar with that of Theorem~\ref{thm:main theorem} and so we only point out the differences. Fix $x_0\in B_{R}$ and $s,t,r$ with $0<t<s\le r<1-R$. Let $\eta\in C_0^\infty(\Omega)$ be such that $0< \eta<1$ on $B_s(x_0)\backslash \bar{B}_t(x_0)$, $\eta\equiv 0$ outside $B_s(x_0)$, $\eta\equiv 1$ on $\bar{B}_t(x_0)$, $|D\eta|\leq c(s-t)^{-1}$, $\{\eta=\frac{1}{2}\}$ has zero Lebesgue measure and $\int_{B_1}\frac{1}{(1-2\eta(x))^4}dx<\infty$. The only difference with the proof of Theorem~\ref{thm:main theorem} is that we will compare the mapping $u$ with the mapping $u_\eta=(1-\eta(x))u(x)+\eta(x)h(x)$.

By similar computation, we  obtain that for any smooth function $\eta\in C_0^\infty(\Omega)$ which satisfies either $0\leq \eta<\frac{1}{2}$ or $\frac{1}{2}<\eta\leq 1$, it holds
\begin{equation}\label{eq:**}
\pi_{u_\eta}\leq (1-\eta)\pi_{u}+\eta\pi_h+C|\nabla \eta|d(u,h)(|\nabla u|_1+|\nabla h|_1)-\nabla\eta\otimes \nabla d^2(u,h)+Q(\eta,\nabla\eta).
\end{equation}
Still set $\mu=n/2$. Taking the trace on both sides of~\eqref{eq:**} and then the $\mu$-th power on both sides, we get
\begin{equation*}\begin{aligned}
\pi_{u_\eta}^\mu & \leq c(n)\Big( (1-\eta)^\mu\pi_{u}^\mu+\eta^\mu|\nabla h|^\mu+|\nabla \eta|^\mu d^\mu(u,h)(\nabla u|_\mu+|\nabla h|_\mu)\\
&\quad-(\nabla \eta\otimes \nabla d^2(u,h))^\mu +Q(\eta,\nabla\eta)^\mu\Big).
\end{aligned}\end{equation*}
Using the $Q$-quasi-$n$-harmonic condition as in the previous proof, we obtain
\begin{equation*}\begin{aligned}
	\int_{B_s(x_0)}(|\nabla u|^2)^\mu dx &\leq C\Big(\int_{B_{s}(x_0)-B_{t}(x_0)}(|\nabla u|^{2})^{\mu}dx+\int_{B_s(x_0)}|\nabla h|^ndx\\
	&\quad+\frac{1}{(s-t)^{\mu}}\int_{B_{s}(x_0)}d^{\mu}(u,h)(|\nabla u|_{\mu}+|\nabla h|_\mu)dx
	+\int_{B_{s}(x_0)}gdx\Big).
\end{aligned}\end{equation*}
Applying the Young's inequality as before, we  deduce
\[
\begin{aligned}\int_{B_{t}(x_0)}(|\nabla u|^{2})^{\mu}dx &\le\theta\int_{B_{s}(x_0)}(|\nabla u|^{2})^{\mu}dx +C\int_{B_s(x_0)}|\nabla h|_n dx\\
&\quad+C\left\{ \frac{1}{(s-t)^{n}}\int_{B_{s}(x_0)}d^{n}(u,h)dx+\int_{B_{s}(x_0)}gdx\right\}
\end{aligned}
\]
for some $C>1$ and $\theta<1$. Then it follows that
\begin{equation}\label{eq: growth estimate 2}
\int_{B_{r/2}(x_0)\cap B_R^+}(|\nabla u|^{2})^{\mu}dx  \le C\left(r^{-n}\int_{B_{r}(x_0)\cap B_R^+}d^{n}(u,h)dx+\int_{B_{r}(x_0)\cap B_R^+}(g+|\nabla h|_n)dx\right).
\end{equation}

We first assume that the $n$-th component of $x_0$ is no bigger than ${3r}/{4}$. In this case, $d(u,h)=0$ in $B_r(x_0)\backslash B_R^+$ and $\mathcal{L}^n\big(B_r(x_0)\backslash B_R^+\big)\geq c\mathcal{L}^n\big(B_r(x_0)\big)$ for some $c=c(n)>0$. By the Sobolev-Poincar\'e inequality and H\"older's inequality, for $n/2< \hat{q}<n$, we have
\begin{equation*}
\begin{aligned}
\dashint_{B_{r}(x_0)\cap B_R^+}d^{n}(u,h)dx&\le C r^n\left(\dashint_{B_{r}(x_0)}\left((|\nabla u|^2)^{\frac {\hat{q}}{2}}+(|\nabla h|^2)^{\frac {\hat{q}}{2}}\right)dx\right)^{\frac {n}{\hat{q}}}\\
&\leq Cr^n\left\{\left(\dashint_{B_{r}(x_0)}(|\nabla u|^2)^{\frac {\hat{q}}{2}}dx\right)^{\frac {n}{\hat{q}}}+\dashint_{B_{r}(x_0)}|\nabla h|_n dx\right\}.
\end{aligned}
\end{equation*}
From now on, the remaining proof for this case has no much difference from that of Theorem \ref{thm:main theorem} and so we omit the details.

If the $n$-th component of $x_0$ is no less than ${3r}/{4}$, we directly apply the proof of Theorem \ref{thm:main theorem} to obtain interior regularity estimate of type \eqref{eq: RHI for w} in which the balls $B_{R/2}$ and $B_{R}$ are replaced by $B_{r/2}(x_0)$ and $B_{3r/4}(x_0)$, respectively. This again implies \eqref{eq: boundary RHI}.  The proof is complete.


\end{proof}

\subsection*{Acknowledgements}

The authors would like to thank Prof.~Stefan Wenger for his interest in this work and for his useful comments. They also wish to thank the referees for their valuable comments that greatly improve the exposition.

\end{document}